\documentclass[12pt]{article}
\usepackage{amsmath,amstext,amssymb,graphicx,bm,enumerate}

\usepackage[colorlinks, citecolor=blue, linkcolor=blue]{hyperref}
\usepackage{xcolor}

\makeatletter

\@addtoreset{equation}{section}

\makeatother

\newtheorem{ex}{\noindent Example}
\newtheorem{theorem}{\noindent Theorem}
\newtheorem{lemma}{\noindent Lemma}
\newtheorem{corollary}{\noindent Corollary}
\newtheorem{proposition}{\noindent Proposition}
\newtheorem{remark}{\noindent Remark}
\newtheorem{definition}{\noindent Definition}
\newenvironment{proof}[1][\proofname]{\par
  \normalfont
  \trivlist
  \item[\hspace{14pt}
        \itshape
    #1{.}]\ignorespaces
}{%
  \endtrivlist
}
\newcommand{\proofname}{Proof}
\renewcommand{\theequation}{\arabic{section}.\arabic{equation}}
\renewcommand{\thetheorem}{\arabic{section}.\arabic{theorem}}
\renewcommand{\thelemma}{\arabic{section}.\arabic{lemma}}
\renewcommand{\thecorollary}{\arabic{section}.\arabic{corollary}}
\renewcommand{\theremark}{\arabic{section}.\arabic{remark}}

\renewcommand{\thedefinition}{\arabic{section}.\arabic{definition}}
\numberwithin{equation}{section}

 \DeclareMathAlphabet{\mathpzc}{OT1}{pzc}{m}{it}

\newcommand{\dif}{\mathrm{d}}

\newcommand{\abs}[1]{\left\vert#1\right\vert}
\newcommand{\set}[1]{\left\{#1\right\}}

\newcommand{\norm}[1]{\left\Vert#1\right\Vert}

\newcommand{\N}{\mathbb{N}}
\newcommand{\R}{\mathbb{R}}
\newcommand{\Z}{\mathbb{Z}}

 \newcommand{\innp}[1]{\langle {#1}\rangle}
\newcommand{\Be}{\begin{equation}}
\newcommand{\Ee}{\end{equation}}
\newcommand{\Bs}{\begin{split}}
\newcommand{\Es}{\end{split}}
\newcommand{\Bes}{\begin{equation*}}
\newcommand{\Ees}{\end{equation*}}
\newcommand{\BT}{\begin{thm}}
\newcommand{\ET}{\end{thm}}
\newcommand{\Bp}{\begin{proof}}
\newcommand{\Ep}{\end{proof}}
\newcommand{\BL}{\begin{lem}}
\newcommand{\EL}{\end{lem}}
\newcommand{\BP}{\begin{proposition}}
\newcommand{\EP}{\end{proposition}}
\newcommand{\BC}{\begin{corollary}}
\newcommand{\EC}{\end{corollary}}
\newcommand{\BR}{\begin{rem}}
\newcommand{\ER}{\end{rem}}
\newcommand{\BD}{\begin{defn}}
\newcommand{\ED}{\end{defn}}
\newcommand{\BI}{\begin{itemize}}
\newcommand{\EI}{\end{itemize}}

\newcommand{\FH}{\mathfrak{H}}
\newcommand{\GH}{\mathfrak{G}}

\begin{document}

\title{
\vspace{-12pt} Sturm-Liouville Theory and Decay Parameter for Quadratic Markov Branching Processes }

\author{
Anyue Chen\thanks{Department of Mathematics, Southern
University of Science and Technology, Shenzhen, Guangdong, 518055, China. and
Department of Mathematical Sciences, The
University of Liverpool, Liverpool, L69 7ZL, \textsc{UK}. \ {\tt
achen@liv.ac.uk}} \  \ Yong Chen \thanks{
 School of Mathematics and Statistics, Jiangxi Normal University, Nanchang,  P. R. China. {\tt zhishi@pku.org.cn }(Corresponding Author)} \  \ Wu-Jun Gao \thanks{College of Big Data and Internet, Shenzhen Technology University, Shenzhen, P. R. China.{\tt gaowujun@sztu.edu.cn}}\  \ Xiaohan Wu \thanks{Department of Mathematics, Harbin Institute of Technology, Harbin, 150001, China, and Department of Mathematics, Southern University of Science and Technology, Shenzhen, Guangdong, 518055, China.{\tt 11849455@mail.sustech.edu.cn}}}

\date{}
\maketitle


\vspace{-24pt}

\begin{abstract}

For a quadratic Markov branching process (QMBP), we show that the decay parameter is equal to the first eigenvalue of a Sturm-Liouville operator associated with the PDE that the generating function of the transition probability satisfies. The proof is based on the spectral properties of the Sturm-Liouville operator. Both the upper and lower bounds of the decay parameter are given explicitly by means of a version of Hardy inequality. Two  examples are provided to illustrate our results. The important quantity, the Hardy index, which is closely linked with the decay parameter of QMBP, is deeply investigated and estimated.


\end{abstract}

\section{Introduction}
\setcounter{section}{1}
 \setcounter{equation}{0}
 \setcounter{theorem}{0}
 \setcounter{lemma}{0}
 \setcounter{corollary}{0}
 \setcounter{remark}{0}
 \setcounter{proposition}{0}
\renewcommand{\theequation}{\arabic{section}.\arabic{equation}}
\renewcommand{\thetheorem}{\arabic{section}.\arabic{theorem}}
\renewcommand{\thelemma}{\arabic{section}.\arabic{lemma}}
\renewcommand{\thecorollary}{\arabic{section}.\arabic{corollary}}
\renewcommand{\theremark}{\arabic{section}.\arabic{remark}}

The motivation of the present paper is to study the decay property of the Quadratic Markov Branching Processes. 
We give the formal definition as follows.
\begin{definition}
A quadratic Markov branching processes is a continuous-time Markov chain with state space $\Z_{+}=\set{0,1,\dots}$ determined by the
 $q$-matrix $Q=\{q_{ij};~i,j\in \Z_{+}\}$ defined by
\begin{equation}\label{branching rate}
   q_{ij}=\left\{
      \begin{array}{ll}
    i^{2}b_{j-i+1} , &\quad\text{if } j\geq i-1\\
    0  , &\quad   \mbox{otherwise.}
      \end{array}
\right.
\end{equation}
where $\{b_j: j\in \Z_{+}\}$ is a given real sequence which satisfies the usual non-trivial conditions
\begin{align}\label{assmup series bj}
  b_j\geq 0~(j\neq 1),\,\, -b_1=\sum_{j\neq 1}b_j,\,\,b_0>0\, \text{ and }\, \sum_{j=2}^{\infty}b_j>0.
\end{align}
\end{definition}
Let $m_d$ and $m_b$ be the mean death and mean birth rates, respectively. Then we have
\begin{equation}
m_d=b_0\qquad \mbox{and}\qquad m_b=\sum_{j=2}^{+\infty}(j-1)b_j.
\end{equation}
When $m_d\ge m_b$, the jump chain almost surely hits the absorbing zero state. Thus, there is a unique Q-function. Uniqueness
may not hold if $m_d<m_b$, but in all cases, the forward Kolmogorov system has exactly one solution, which is the Feller minimal solution, see Chen \cite{Chen2002a}. The corresponding Markov process $\set{Z(t);\,t\ge 0}$ is called a quadratic Markov Branching process, henceforth referring to as a QMBP. Throughout this paper, we shall assume that $m_d\ge m_b$.
\par Let \begin{equation}\label{1.4}
B(s)=\sum_{j=0}^{\infty}b_js^j,
\end{equation}
denotes the generating function of the sequences $\{b_j;j\geq 0\}$. As power series, this generating function has a convergence radius $\varrho_b^{-1}=\limsup_{n\to\infty}\sqrt[n]{b_n}$. Clearly, $\varrho_b\geq 1$.
\par The generating function $B(s)$ possesses the following simple yet useful properties whose proof is well known and thus omitted here.

\begin{proposition}\label{prop 1.1}
The generating function $B(s)$ is a convex function of $s\in [0, \varrho_b)$ and hence the equation $B(s)=0$ has at most two roots in $[0, \varrho_b)$ and, in particular, in $[0,1]$. More specially, if $B'(1)\leq 0$ then $B(s)>0$ for all $s\in[0,1]$ and $1$ is the only root of the equation $B(s)=0$ in $[0,1]$ while if
$B'(1)>0$ (including $B'(1)=+\infty$) then the equation $B(s)=0$ has exactly two roots $q$ and $1$ with $0<q<1$ such that $B(s)>0$ for $0\leq s<q$ and $B(s)<0$ for $q<s<1$.
\end{proposition}

\par
It is easy to see that $B'(1)= m_b-m_d$ which explains the probability meaning of the important quantity $B'(1)$.
\par Let $P(t)=(P_{ij}(t))$ denote the transition function where $P_{ij}(t)=P(Z_t=j\,|\,Z_0=i).$ Denote the communicating class for the transition function $P(t)$ as $C$. By the assumption given in Definition $1.1$, it is easy to see that for our quadratic Markov branching processes, the communicating class $C$ is just $\{1,2,\cdot\cdot\cdot\}$. The decay parameter of the process is defined by
\begin{equation}
\lambda_C=-\lim_{t\to \infty}\frac{1}{t}\log P_{ij}(t).
\end{equation} General theory asserts the existence of the limit and that it is independent of $i,j\in C$. It
is easy to show that
\begin{align}
\lambda_C=\inf\set{\lambda \ge 0,\,\, \int_0^{\infty} P_{ij}(t) e^{\lambda t}\dif t =\infty,\,\, i,j\in C}.
\end{align} For a review of this topic, we refer the readers to van Doorn and Pollett\cite{van pollett 11}.
A very useful representation of the decay parameter can be found in Theorem 3.3.2(iii) of Jacka and Roberts\cite{jacka}.
\par In nearly all the stochastic models which can be well modelled by a continuous time Markov chain with absorbing states, obtaining and/or estimating the corresponding decay parameter is a very important topic. The main aim of this paper is thus to investigate this very important question for QMBP.

The structure of this paper is as follows: after the introductory Section 1, we state our main conclusions obtained in this paper in Section 2 and the proofs will be given in Sections 3 and 5, respectively. Examples will be provided in Section 4.
\setlength{\abovedisplayskip}{3pt} 

\section{Main Results}
\setcounter{section}{2}
 \setcounter{equation}{0}
 \setcounter{theorem}{0}
 \setcounter{lemma}{0}
 \setcounter{corollary}{0}
 \setcounter{remark}{0}
 \setcounter{proposition}{0}
\renewcommand{\theequation}{\arabic{section}.\arabic{equation}}
\renewcommand{\thetheorem}{\arabic{section}.\arabic{theorem}}
\renewcommand{\thelemma}{\arabic{section}.\arabic{lemma}}
\renewcommand{\thecorollary}{\arabic{section}.\arabic{corollary}}
\renewcommand{\theremark}{\arabic{section}.\arabic{remark}}

Our first main result is a representation theorem of the decay parameter $\lambda_{C}$ of the QMBP by means of the classical generating function method. 
Let $\{F_i (s, t); i \in \Z_+\}$ be the generating functions of $Q$-function $P(t)$ of the QMBP. That is
$$F_i(s, t)=\sum_{j=0}^\infty P_{ij}(t)s^j, \qquad (i\geq0).$$
Denote
\begin{equation}\label{weight func.}
w(s)=\frac{1}{B(s)} ,\quad J=(0,1),\end{equation}
where $B(s)$ is defined in (\ref{1.4}).\\
Consider the differential expression $M$ defined by
\begin{equation} \label{M express}
My:=(-s y'(s))',\quad y\in \FH=L^2(J,\,w). \end{equation}
It is known by Chen \cite{Chen2002a} that $F_i(s,t)$ is the unique solution of the equation
\begin{equation}\label{pde 1}
     \frac{\partial }{\partial t}  {F}_i(s,t)=-w^{-1} M F_i(s,t) , \qquad (t,s)\in (0,\infty)\times (0,1)
    \end{equation}
 with initial condition
 $$ F_i(s, 0)= s^i.$$

To solve the PDE (\ref{pde 1}), we will make use of the Sturm-Liouville theory. We firstly find the suitable self-adjoint realization $(S,\,D(S))$ of the minimal operator $S_{min}$ of $(M,w)$ on $J$ (see Definition~\ref{def maxmin} below),  and then study the spectral properties of  $(S,\,D(S))$. The  following is our representation theorem of $\lambda_C$ for QMBP.

\begin{theorem}\label{thm main}
 If 
  $B'(1)<0$, 
 then the decay parameter $\lambda_C$ for QMBP is equal to the first eigenvalue $\ell_0$ of the self-adjoint Sturm-Liouville operator $(S,D(S))$ defined by 
\begin{align}
S g&= w^{-1} M g,\quad \text{ for  } g\in D(S),\\
D(S)&=\set{y+c v_1:\,y\in D_{{min}},\,c\in \R}, \label{domain d of s}
\end{align}where $D_{min}$ is the domain of $S_{min}$, and $v_1$ is a $C^{\infty}(J)$ function such that
\begin{equation}\label{function u}
v_1(s)=  \left\{
      \begin{array}{ll}
    1 , &\quad\text{when  } 0<s< c_1, \\
     0, &\quad\text{when  } c_2<s<1.
      \end{array}
\right.
\end{equation}
with $0<c_1<c_2<1$.
\end{theorem}
{\begin{remark}
The identity \eqref{domain d of s} means that the dimension of the quotient space $D(S)/D_{min}$ of $D(S)$ and $D_{min}$ is $1$. That is to say, the deficiency index of the differential expression $M$ on $J$ is $d=1$. The function $v_1$ is not unique and can be taken as any function $\tilde{v}$ such that $\tilde{v}-v\in D_{min}$. \end{remark}}

By Theorem 2.1, to find the decay parameter $\lambda_C$ for QMBP is just to find the first eigenvalue $\ell_0$ of the self-adjoint operator $(S, D(S))$. Now our second main result is just the variational formula of the first eigenvalue $\ell_0$ which gives a better upper and lower bound of $\lambda_{C}$.
\begin{theorem}\label{prop bound b2}
If $B'(1)<0$, then the variational formulae of the first eigenvalue $\ell_0$ of the operator $(S,D(S))$(i.e. $\lambda_C$) is
\begin{align}\label{variational formula202008}
\ell_0=\inf\set{\frac {\int_0^1 s\big( g'(s)\big)^2 \dif s}{\int_0^1 g^2(s) w(s)\dif s}:\,\, g\not\equiv 0,\,g\in C_{c}^{\infty}(J)}.
\end{align} Furthermore, $\ell_0$(i.e. $\lambda_C$) has lower and upper bounds
\begin{equation}\label{lambda0 bound}
 \frac{1}{ 4 D^2}\leq \ell_0\leq  \frac{1}{ D^2},
\end{equation}
where $D^2$ is given by
\begin{equation}\label{d2-square}
   D^2:=\sup_{s\in (0,1)}\set{(-\log s )\cdot \big(\int_0^s \frac{1}{B(r)}\dif r \big )}.
\end{equation}
\end{theorem}
By Theorems 2.1 and 2.2, particularly by (\ref{lambda0 bound}) and (\ref{d2-square}), we see that to estimate the value of $D^2$ is a key issue. Let's agree to call $D^2$ as the Hardy index. The following corollaries then concentrate on discussing this Hardy index.
\vspace{.5cm}
\begin{corollary}\label{coro 2.1}\begin{itemize}
 \item[]The quantity $D^2$ in (\ref{d2-square}) satisfies
\begin{align}\label{eq2.11}
\frac{(\log 2)^2}{b_0  }\le D^2\le \frac{(\log 2)^2}{b_0-m_b},
\end{align}which implies
$$ \frac{b_0-m_b}{4(\log 2)^2}\le \lambda_C\le \frac{b_0}{(\log 2)^2}. $$

    \end{itemize}
\end{corollary}
A more sharp bounds for $\lambda_C$ can be given as follows.
\begin{corollary}\label{coro 2.2}
 The quantity $D^2$ in (\ref{d2-square}) further satisfies
\begin{align}\label{eq2.12}
\frac{\kappa_2(\log(1+\sqrt{\kappa_2}))^2}{m_b-\kappa_2 }\le D^2\le \frac{(\log(1+\sqrt{\kappa_1}))^2}{b_0-m_b},
\end{align}
where $\kappa_1=\frac{m_b}{b_0}$, $\kappa_2=\frac{m_b}{A(s_0)+s_0\cdot m_b}$ and A(s) is determined by $A(s)=\frac{B(s)}{1-s}$ and $s_0$ is determined by the equation $-m_b=A'(s_0)$ which guarantees that $0<s_0<1$ and that $\kappa_1<\kappa_2$.
Therefore we have that \begin{align}\label{eq2.13}
\frac{b_0-m_b}{4(\log(1+\sqrt{\kappa_1}))^2}\le\lambda_C\le\frac{m_b-\kappa_2 }{\kappa_2(\log(1+\sqrt{\kappa_2}))^2}.
\end{align}
\end{corollary}

\begin{corollary}\label{coro 2.2*}
If $B''(1)<2b_0$, then the quantity $D^2$ in (\ref{d2-square}) also satisfies
\begin{align}\label{eq2.12*}
\frac{(\log(1+\sqrt{\kappa'_2}))^2}{b_0-m_b}\le D^2\le \frac{(\log(1+\sqrt{\kappa'_1}))^2}{b_0-m_b},
\end{align}
where $\kappa'_1=\frac{B''(1)}{2b_0}$, $\kappa_2=\frac{\sum_{j=2}^{+\infty}b_j}{b_0}$
which implies that \begin{align}\label{eq2.13*}
\frac{b_0-m_b}{4(\log(1+\sqrt{\kappa'_1}))^2}\le\lambda_C\le\frac{b_0-m_b}{(\log(1+\sqrt{\kappa'_2}))^2}
\end{align}
\end{corollary}
\begin{remark}
There are two kinds of bounds of the quantity $D^2$ in (\ref{d2-square}) in Corollaries \ref{coro 2.2} and \ref{coro 2.2*}, respectively. It could be easily obtained that the upper bound in  Corollary \ref{coro 2.2} is better than the one in  Corollary \ref{coro 2.2*} if and only if $m_b<\frac12B''(1)$. Also the lower bound in  Corollary \ref{coro 2.2} is better than the one in \ref{coro 2.2*} if and only if $\frac{\log(1+\sqrt{1+k_2})}{\log(1+\sqrt{k'_2})}>\frac{A(s_0)+m_b\cdot s_0-m_b}{b_0-m_b}$. Furthermore, the assumption $B''(1)<2b_0$ is not necessary for the upper bound of $D^2$, i.e. the lower bound of $\lambda_C$ in Corollary \ref{coro 2.2*}.

\end{remark}
\par Furthermore, we may found a new and better upper and lower bound of $D^2$ by using the result for Example 2 discussed in Section 4.
\begin{corollary}\label{coro 2.3}
There exist $s_1\in(0,1)$ and $s_2\in(0,1)$ such that the quantity $D^2$ in (\ref{d2-square}) satisfies
\begin{equation}\label{eq 2.16}
\phi_1(s_1)\leq D^2\leq \phi_2(s_2),
\end{equation}
where $\phi_1(s)=(-\log s)\int_0^s \frac{\dif r}{(1-r)[b_0+(b_0+b_1)r+\frac12 A''(0)r^2]}$,  $\phi_2(s)=(-\log s)\int_0^s \frac{\dif r}{(1-r)[b_0+(b_0+b_1)r+\frac12 A''(1)r^2]}$, $s_1$ and $s_2$ are determined by $\sup_{s\in(0,1)}\phi_1(s)=\phi_1(s_1)$ and $\sup_{s\in(0,1)}\phi_2(s)=\phi_2(s_2)$, respectively.
\par Furthermore,
\begin{equation}\label{eq2.16}
\frac{1}{4\phi_2(s_2)}\le \lambda_C\le\frac{1}{\phi_1(s_1)}.
\end{equation}
\end{corollary}
The proofs of these four corollaries can be found in Section 5.

\section{The Sturm-Liouville theory and the proofs of Theorems \ref{thm main} and \ref{prop bound b2}}

\setcounter{section}{3}
 \setcounter{equation}{0}
 \setcounter{theorem}{0}
 \setcounter{lemma}{0}
 \setcounter{corollary}{0}
 \setcounter{remark}{0}
 \setcounter{proposition}{0}
 \setcounter{definition}{0}
\renewcommand{\thedefinition}{\arabic{section}.\arabic{definition}}
\renewcommand{\theequation}{\arabic{section}.\arabic{equation}}
\renewcommand{\thetheorem}{\arabic{section}.\arabic{theorem}}
\renewcommand{\thelemma}{\arabic{section}.\arabic{lemma}}
\renewcommand{\thecorollary}{\arabic{section}.\arabic{corollary}}
\renewcommand{\theremark}{\arabic{section}.\arabic{remark}}
As stated before, throughout this section, we always assume that $B'(1)<0$, that is that we assume that $m_d > m_b$.
\subsection{Strum-Liouville theory}

\par
 For the differential expression $M$ given by
\begin{equation*}\label{eigen-equa}
   My(s):=-(p(s)y'(s))'+ q(s)y(s),\quad \quad\lambda\in \R, \quad \text{ on }\quad  J,
 \end{equation*}  with
 \begin{equation}\label{eigen-equa condition}
 J=(a,b),\,-\infty\le a<b\le \infty,\quad 1/p, q, w\in L_{\rm{loc}} (J,\,\R),
 \end{equation}
 and the expression domain of $M$ being functions $y$ such that
$y,\, py'\in AC_{loc}(J)$, 
the following definitions are taken from Zettl\cite{az1}.
\begin{definition} \label{def maxmin}(The maximal and minimal operators) The maximal domain $D_{max}$ of $M$ on $J$ with weight function $w>0$ is defined by
$$D_{max}=\set{g\in L^2(J,w):\,g,pg'\in AC_{loc}(J),w^{-1}Mg\in L^2(J,w) }.$$
Define
\begin{align*}
S_{max} g&= w^{-1}Mg, \text{ for }  g\in D_{max},\\
S_{min} ' g&= w^{-1}Mg, \text{ for }  g\in D_{max}, \,g \text{ has compact support in  } J.
\end{align*}Then $S_{max}$ is called the maximal operator of $(M,w)$ on $J$, $S_{min}'$ is called the preminimal operator and the minimal operator $S_{min}$ of $(M,w)$ on $J$ is defined as the closure of $S_{min}'$. The domain of $S_{min}$ is denoted by $D_{ min}$.
\end{definition}
Any self-adjoint extension of the minimal operator  $S_{min}$ satisfies  $$S_{{min}}\subset S=S^*\subset S_{{max}}.$$
It is well-known that the domain $D(S)$ is determined by two-point boundary conditions which depend on the limit-circle/limit-point classification of the endpoints.

\begin{definition}\label{def lclp}
Consider the Sturm-Liouville equation
\begin{equation}\label{eigen-equa2} My(s)=\ell w(s) y(s). \end{equation}
The endpoint $a$
\begin{itemize}
\item  is regular if, in addition to (\ref{eigen-equa condition}),
\begin{equation*}
1/p, q, w\in L ((a,d),\,\R)
\end{equation*} holds for some  (and hence any)  $d\in J$;
\item  is limit-circle (LC) if all solutions of the equation (\ref{eigen-equa}) are in $L^2((a,d),\, {w})$ for some (and hence any) $d\in (a,b)$;
\item  is limit-point (LP) if it is not limit-circle.
\end{itemize}
Similar definitions are made at endpoint $b$.  An endpoint is called singular if it is not regular.  It is well-known that the LC, LP, classification are independent of $\ell\in \R$.
\end{definition}
\begin{lemma}\label{lem e.0802}
Let $(M,w)$ be given as in (\ref{weight func.}) and (\ref{M express}). Then both $s=0$ and $s=1$ are singular, and the endpoints $s=0$ and $s=1$ are limit circle and limit point respectively. Moreover,  the deficiency index of $M$ on $J$ is $d=1$.
\end{lemma}
\begin{proof}
It is clear that $\frac{1}{s}\notin L((0,d), \R)$ and $w\notin L((d,1), \R)$. Hence, The endpoints $s=0$ and $s=1$ are singular.

Let $\bar{v}_1(s) \equiv 1$ and $v_2(s)=\log s$ on (0,1).  Taking $\ell=0$, it is easy to see that $\bar{v}_1,\,v_2$ are nontrivial linearly independent solutions of the equation
$$M y(s)=(-s y'(s))'=  \ell w(s) y(s) .$$
Since $w(s)=\frac{1}{(1-s)A(s)}$ where $A(s)> 0$ and analytic on $[0,1]$ (see the following Lemma~\ref{lem 5.14}), we see that $\bar{v}_1,\,v_2 \in L^2((0,d),\, {w}))$ and $\bar{v}_1\notin L^2((d,1),\, {w}))$ with $d\in (0,1)$. Hence, by Definition~\ref{def lclp}, the endpoints $t=0$ and $t=1$ are limit circle and limit point respectively. Hence, the deficiency index of $M$ on $J$ is $d=1$, see Theorem 10.4.5 of Zettl\cite{az1}.
\end{proof}{\hfill\large{$\Box$}}


\begin{lemma}\label{prop202008}
Let  $v_1\in C^{\infty}(J)$ be given as in (\ref{function u}). 
Then
\begin{align}
D(S)
&=\set{y\in D_{{max}}:\, \lim_{s\to 0+} s y'(s)=0}\label{domain S}\\
&=\set{y+c v_1:\,y\in D_{{min}},\,c\in \R}. \label{domain S 2}
\end{align} is a self-adjoint domain. 
Moreover, $(S,D(S))$ is the unique self-adjoint extension of $S_{{min}}$ such that $y(s)=s-1$ belongs to the domain $D(S)$.
\end{lemma}
\begin{proof}
$v_1$ can be constructed by means of the smooth cut-off function, please refer to Davies\cite[p.47]{davies 1995} for details.
Let $v_2(s)=\log s$ and $\ell=0$.

Let $p(s)=s$. For $y$ and $z$ in the expression domain of $M$, the Lagrange sesquilinear form $[\, \, ,\,]$ is given by
\begin{align*}
[y,z]:&=ypz'-zpy'.\end{align*}
It is known that for any $y,z\in D_{max}$ both limits\begin{align*}
[y,z](0)=\lim_{s\to 0+} [y,z](s),\quad [y,z](1)\lim_{s\to 1-} [y,z](s)
\end{align*}exist and are finite. See Zettl\cite[Lemma10.2.3]{az1}.

It is clear that $v_1,\,v_2$ are nontrivial real solution of the equation
$$M y(s)=(-s y'(s))'=  \ell w(s) y(s) $$
on $(0, c_1)$ satisfying that $[v_1,v_2](s)=1,\,s\in (0,\,c_1)$. When $y\in D_{{max}}$, we have that
\begin{align*}
[y,v_1](0)&=\lim_{s\to 0+} [y,v_1](s) =-\lim_{s\to 0+} s y'(s),\\
[y,v_2](0)&=\lim_{s\to 0+} [y,v_2](s) =\lim_{s\to 0+}( y(s)-s\log s y'(s) ).
\end{align*}

Since $0$ is limit circle and $1$ is limit point, Theorem 10.4.5 of Zettl\cite{az1}  says that  $D(S)$ is a self-adjoint domain if and only if
there exist $A_1,A_2\in \R$ with $(A_1,A_2)\neq (0,0)$, such that
$$D(S)=\set{y\in D_{{max}}:\, A_1\cdot[y,v_1](0)+ A_2\cdot[y,v_2](0)=0}$$
holds. Now taking $(A_1,A_2)= (1,0)$, we obtain (\ref{domain S}).

It is easy to check that $v_1,\,v_2\in D_{max}$. Since $[v_1,v_2](0)=1$, we have $v_1\notin D_{min}$.
It is clear that $sv'_1(s)=0$ on $(0,c_1)$. Thus, $v_1\in D(S)/D_{min}$. Note that the deficiency index of $M$ on $J$ is $d=1$. Hence, we obtain (\ref{domain S 2}).

When $y(s)=s-1$, we see that $[y,v_1](0)=0,\,[y,v_2](0)=-1$. If $(A_1,A_2)\neq (0,0)$ satisfies $A_1\cdot [y,v_1](0)+ A_2\cdot [y,v_2](0)=0$ then $A_1\neq 0,\, A_2=0$. Hence, (\ref{domain S}) is the unique self-adjoint extension of $S_{{min}}$ such that $y(s)=s-1$ belongs to the domain $D(S)$.
\end{proof}{\hfill\large{$\Box$}}
\par We show that the operator $(S,D(S))$ has the BD property, i.e., it is spectra discrete and bounded below.
\begin{lemma}\label{spectra property}
   The operator $(S,D(S))$ has the BD property. Moreover, the spectrum
   $\sigma(S)$ is real, simple and discrete;
   \begin{align*}
      \sigma(S)=\set{\ell_k\in \R, k=0,1,2,\dots}\\
      \ell_k<\ell_{k+1},\quad \ell_k\to \infty  \,(\text{as } k\to \infty).
   \end{align*}
If $\varphi_k $ is an eigenfunction of $\ell_k$ then
$\varphi_k \in C^{\infty}(J)$ and has exactly $k$ zeros in $J=(0,1)$. In addition, the set of eigenfunctions $\set{\varphi_k ,k\in \Z_{+} }$ is orthogonal
and complete in $\mathfrak{H}$.
\end{lemma}
\begin{proof}
Denote $$A[\alpha,\beta]=\set{f:[\alpha,\beta]\to \R:\, f\in AC[\alpha,\beta],f'\in L^2(\alpha,\beta), \text{ and } f(\alpha )=f( \beta)=0}.$$
We denote
\begin{equation}\label{3.5}
B(s)=(1-s)A(s).
\end{equation}
Then $A(s)\neq 0$ and is analytic on $\abs{s}< 1$. It follows from Lemma 2.2 of Chen \cite{Chen2002a} that
\begin{align}
   B(s)&>0\qquad \forall s\in [0,1),\nonumber\\
   A(s)&>0\qquad \forall s\in [0,1],\label{bx gt 0}
\end{align} where in the last inequality $A(1)>0$ is from $B'(1)=-A(1)<0$. 

   The proof follows the spirit of Bailey et al.\cite{behz} and Hinton and Lewis\cite{hl}. We need only show that for each real number $\ell$ there is a $\delta>0$ which may depend on $\ell$ so that
   if $[\alpha,\beta]\subset (0,\, \ell)$ or $[\alpha,\beta]\subset (1-\delta, \, 1)$ and $y\in A[\alpha,\beta],\,y\not\equiv 0$, then 
   \begin{equation}\label{bd property}
      \int_{\alpha}^{\beta} \set{s(y'(s))^2-\ell w(s) y^2(s)}\dif s >0.
   \end{equation} It is clear we need only show \eqref{bd property} for $\ell>0$.
We make use of a Hardy-type inequality, see Hinton and Lewis\cite{hl}: if $f\in A[\alpha,\beta]$ with $f \not\equiv 0$, then 
\begin{align}\label{hardy-type 2}
   \int_{\alpha}^{\beta} \frac{1}{s(\log s)^{2}} f^2(s)\dif s\leq 4 \int_{\alpha}^{\beta} s[f'(s)]^2 \dif s.
\end{align}


For any $\ell>0$,  we have that when $s>0$ is small enough,
$$\frac14 \frac{1}{s(\log s)^{2}}-\frac{\ell}{(1-s)A(s)}\geq\frac14 \frac{1}{s(\log s)^{2}}-\frac{\ell}{(1-s)m} >0,$$
where $m>0$ is the minimum value of $A(s)$ on $[0,1]$, and it follows from (\ref{hardy-type 2}) that
\begin{align*}
   \int_{\alpha}^{\beta} \set{s(y')^2-\ell w  y^2} \dif s
   \geq \int_{\alpha}^{\beta}\Big( \frac14 \frac{1}{s(\log s)^{2}}-\frac{\ell}{(1-s)A(s)}\Big) y^2  \dif s  >0.
\end{align*}

The well known inequality $$\frac{x}{1+x} \le  \log(1+x)\le x,\quad \forall x>-1,$$ implies that when $s\in (0,1)$,
\begin{align*}
\frac{1}{s(\log s)^{2}}&=\frac{1}{s\big(\log(1+ s-1)\big)^{2}}\geq \frac{1}{s}\frac{s^2}{(1-s)^2}=\frac{s }{(1-s)^2}.
\end{align*} Hence, for any $\ell>0$,  we have that when $1-s$ is small enough,
\begin{align*}
\frac14 \frac{1}{s(\log s)^{2}}-\frac{\ell}{(1-s)A(s)}&\geq\frac{1 }{1-s } \Big(\frac14\frac{s }{1-s }  -\frac{\ell}{  m}\Big)>0.
 \end{align*}
Together with (\ref{hardy-type 2}), we have that 
\begin{align*}
   \int_{\alpha}^{\beta} \set{s(y')^2-\ell w y^2} \dif s \geq  \int_{\alpha}^{\beta}\Big( \frac14 \frac{1}{s(\log s)^{2}}-\frac{\ell}{(1-s)A(s)}\Big)y^2 \dif s >0.
\end{align*}

Therefore, (\ref{bd property}) holds, which implies that the operator $L$ has BD property. Since the endpoint $s=1$ is limit point, the other conclusions are given by the case (8.ii) of Theorem 10.12.1 in Zettl\cite[p.208]{az1} and Theorem XIII 4.2 of Dunford and Schwartz\cite[p.1331]{dunford-schwartz}.
\end{proof}{\hfill\large{$\Box$}}
\setlength{\abovedisplayskip}{3pt} 

Denote by $\innp{f,\,g}$ the inner product in the space $\FH$ for every pair of elements $f,\,g$ in $\FH=L^2(J,\,w)$.
\begin{lemma}\label{lem.02.0802}
If $f\in D_{{min}}$
then
\begin{equation}\innp{S_{{min}}f,\,f }\ge \int_0^1 s (f'(s))^2\dif s.
\end{equation}
\end{lemma}
\begin{proof}
$f\in D_{{min}}$ implies that there exsits a series $f_n$ with compact support in $J$ such that $f_n\to f$ and
$S_{{min}} f_n\to S_{{min}}f $ in $\mathfrak{H}$.

Hence, for any $0<\epsilon<s<1$, we have that as $n\to \infty$,
\begin{align*}
-sf_n'(s)+\epsilon f_n'(\epsilon)&=\int_{\epsilon}^s (-r f_n'(r))'\dif r\\
& \to \int_{\epsilon}^s (-r f'(r))'\dif r\\
&=-sf'(s)+\epsilon f'(\epsilon).
\end{align*}
Thanks to (\ref{domain S}), by letting $\epsilon\to 0$, we see that $f_n'(s)\to  f'(s)$ holds for all $s\in J$.

Moreover, integration by parts implies that
\begin{align*}
\innp{ S_{{min}}f,\,f }&=\lim_{n\to \infty} \innp{S_{{min}}f_n,\,f_n}\\
&=\lim_{n\to \infty} \int_{0}^1 (-sf_n'(s))' f_n(s)\dif s\\
&=\lim_{n\to \infty}   \int_0^1 s (f_n'(s))^2\dif s\\
&\ge  \int_0^1 s (f'(s))^2\dif s,
\end{align*}where the last inequality follows from the Fatou's lemma.
\end{proof}{\hfill\large{$\Box$}}
\begin{lemma}\label{positive selfadjoint}
$(S,D(S))$ is a 
non-negative self-adjoint operator on $\mathfrak{H}$ .
\end{lemma}
\begin{proof}
Let $v_1$ be given as in (\ref{function u}).  The identity (\ref{domain S 2}) implies that we need only show that
$ \innp{S(f+cv_1), f+cv_1 }\ge 0$ holds for all $f\in D_{{min}}$ and $c\in \R$. For simplicity, we can assume that $c=1$. Lemma~\ref{lem.02.0802} implies that
\begin{align*}
\innp{S(f+v_1), f+v_1 }&=\innp{S f, f  }+ 2\innp{Sf, v_1} +\innp{S v_1, v_1 }\\
&\ge  \int_0^1 s (f'(s))^2\dif s +2\innp{Sf,v_1} +\innp{S v_1, v_1 }.
\end{align*}
By integration by parts, we see that
\begin{align*}
\innp{S v_1, v_1 }=\int_0^1 s (v_1'(s))^2\dif s ,\qquad
\innp{Sf,v_1}=\int_0^1 sf'(s)v_1'(s) \dif s.
\end{align*}Hence,
\begin{align*}
\abs{2\innp{Sf,v_1}}\le \int_0^1 s[(f'(s))^2+(v_1'(s))^2] \dif t=\int_0^1 s (f'(s))^2\dif s +\innp{S v_1, v_1 }.
\end{align*}
Thus, we see that $\innp{S(f+v_1), f+v_1 }\ge 0$.
\end{proof}{\hfill\large{$\Box$}}
\begin{corollary} \label{coro e.200802}
The first eigenvalue of the operator $(S,D(S))$ is positive, i.e., $\ell_0>0$.
\end{corollary}
\begin{proof}Lemma~\ref{positive selfadjoint} implies that $\ell_0\ge 0$. We need only show that $0$ is not an eigenvalue. In Lemma~\ref{lem e.0802}, we have shown that the solutions of the equation $Sf\equiv 0$ are $v_1(s) \equiv 1$ and $v_2(s)=\log s$ on $(0,1)$. It is clear that neither $v_1$ nor $v_2$ is in $D(S)$, which implies that $0$ is not an eigenvalue. Hence, $\ell_0>0$.
\end{proof}{\hfill\large{$\Box$}}

\subsection{Proof of Theorem \ref{thm main}}


We first provide a representation of the generating function ${F}_i(s, t)$ with $i\ge 1$. Since $F_i(s, 0)\notin D(S)$, we can not apply the eigenfunction expansion theory in $\mathfrak{H}$ directly. But it is clear that $F_i(s, 0)-1\in D(S)$.  Hence, 
to get around this difficulty, we need only consider the equation of the function $\bar{F}_i(s, t)=1+F_i(s, t)$. Then we obtain
\begin{equation}\label{pde 2}
     \frac{\partial }{\partial t} \bar{F}_i(s,t)=-S F_i(s,t), \qquad (s,t)\in  (0,1)\times (0,\infty)
    \end{equation}
 with initial condition
    $$     \bar{F}_i(s,\,0)=s^i-1.$$

We will derive a series representation of $\bar{F}_i(s, t)$ by the eigenfunction method.



\begin{lemma}\label{prop fitx} In the sense of abstract Cauchy problems, the above PDE (\ref{pde 2}) has a unique solution (one and only one solution) whose eigenfunction expansion is:
   \begin{equation}\label{fitx}
   \bar{F}_i(s,t)=\sum_{k=0}^{\infty} a^{(i)}_k e^{-t\ell_k} \varphi_k(s),\qquad s\in (0,1)
\end{equation}
where the series converges in $L^2(J,w)$, $\set{\ell_k,\,\varphi_k(s)}$ are the spectra of the operator $(S,D(S))$ given in Lemma~\ref{spectra property} and the coefficient
$\set{a^{(i)}_k }$ is given 
by
\begin{equation}
    a^{(i)}_k =\innp{s^i-1,\, \varphi_k(s)}.
\end{equation}
\end{lemma}
\begin{proof}
   We resort to the theory of semigroups of linear operators, see Pazy\cite[chapter 4]{pazy}.

   First, by Lemma~\ref{spectra property} and Corollary~\ref{coro e.200802}, the Hille-Yosida theorem, see Pazy\cite[Theorem 1.3.1]{pazy},  implies that $(-S,D(S))$ is the infinitesimal generator of a
   $C_0$ semigroup of contractions $\set{T(t),t\geq 0}$ on $L^2(J,w)$.

   Second, it follows from Pazy\cite[Theorem 4.1.3]{pazy} that the abstract Cauchy problem (\ref{pde 2}) has a unique solution $u(t)=T(t)f$ for every initial value $f\in D(S)$. Taking $f=s^i-1$, we have that $\bar{F}_i(s,t)=T(t)f$.

   Third, by the spectral theorem of self-adjoint operator, the solution $\bar{F}_i(s,t)$
   has the following representation,
   \begin{align}
      \bar{F}_i(s,t)&=\sum_{k=0}^{\infty} e^{-t \ell_k}\varphi_k(s) \innp{f,\,\varphi_k} =\sum_{k=0}^{\infty}  a^{(i)}_k e^{-t\ell_k}\cdot \varphi_k(s). \label{Tt f}
   \end{align}
\end{proof}{\hfill\large{$\Box$}}

The following lemma ensures that the series in (\ref{fitx}) can be differentiated with respect to $s$ term by term.
\begin{lemma}\label{unif conveg}
For each $t\in[0,\infty)$, the series $\sum_{k=0}^{\infty} a^{(i)}_k e^{-t\ell_k} \varphi_k(s)$ and $\sum_{k=0}^{\infty} a^{(i)}_k e^{-t\ell_k} \varphi_k'(s)$ converge absolutely and uniformly with respect to $s$ in compact subset of $J=(0,1)$, where $\varphi_k{'}$ means the derivative of $\varphi_k$.
\end{lemma}
\begin{proof}

Since $f=s^i-1\in D(S)$, we have that $T_tf \in D(S)$ from Theorem 2.4 (c) of Pazy\cite[p5]{pazy}. Note that the second order differential operator $(S,D(S))$ has a complete orthonormal set $\set{\varphi_k}$ of eigenfunctions. Thus, Theorem XIII 4.3 of Dunford and Schwartz\cite[p1332]{dunford-schwartz} implies that the eigenfunction expansion
$$T_tf (s)= \sum_{k=0}^{\infty} a^{(j)}_k e^{-t\ell_k} \varphi_k(s) $$
converges uniformly and absolutely on each compact subinterval of $J=(0,1)$, and the series may be differentiated term by term with the differentiated series retaining the properties of absolute and uniform convergence.
\end{proof}{\hfill\large{$\Box$}}
\begin{lemma}\label{pijt expression}
   For any $i\in \N$ and for each $t\in[0,\infty)$,  we have that
\begin{equation}\label{pijt}
   P_{i1}(t)+\sum_{j=2} j P_{ij}(t)s^{j-1}= \sum_{k=0}^{\infty}  a^{(i)}_k e^{-t\ell_k} \varphi_k{'}(s),\qquad   s\in (0,1).
\end{equation}
\end{lemma}
\begin{proof}
The uniqueness of the solution to the PDE (\ref{pde 2}) implies that
\begin{align}\label{representation aaa}
  \sum_{j=0}^{\infty} P_{ij}(t)s^{j}= F_i(s,t)=1+\sum_{k=0}^{\infty}  a^{(i)}_k e^{-t\ell_k} \varphi_k(s),\quad s\in(0,1).
\end{align}
Because the series on the left hand side of (\ref{representation aaa}) is an analytic function of $s$ when $\abs{s}< 1$ and the series on the right hand side of (\ref{representation aaa}) can be differentiated about $s\in (0,1)$ term by term, so we can differentiate term by term with respect to $s$ the two series at (\ref{representation aaa}).
\end{proof}{\hfill\large{$\Box$}}
\begin{remark}
   We can characterize the decay parameter $\lambda_C$ just using Eq.(\ref{pijt}). That is to say, we do not need to take $s\to 0+$ in Eq.(\ref{pijt}) to obtain an explicit expression of $P_{i1}(t) $ as in previous work of Letessier and Valent \cite{lv1} and Roehner and Valent\cite{rv}. 
\end{remark}

\begin{lemma}\label{lda great}
   If $B'(1)<0$, then the decay parameter $\lambda_C$ for QMBP satisfies the inequality
   \begin{equation}
      \lambda_C\geq \ell_0.
   \end{equation}
\end{lemma}
\begin{proof}
By taking $t=0$ in Lemma~\ref{unif conveg}, we see that the series $\sum_{k=0}^{\infty} a^{(i)}_k \varphi{'}_k(s) $ is uniformly and absolutely convergent on every compact subset of $I=(0,1)$.
Thus, by the Weierstrass M-test, the series $\sum_{k=0}^{\infty} a^{(i)}_k e^{-t(\ell_k-\lambda)} \varphi_k'(s)$ is uniformly convergent with respect to $t\in[0,\infty)$ for each $s\in (0,1)$.

Observe that $P_{11}(t)$ is dominated by the left-hand side of (\ref{pijt}), so taking the Laplace transform and integrating term by term we obtain the bound for each ${\lambda}<\ell_0$,
\begin{align*}
   \int_{0}^{\infty} e^{\lambda t} P_{11}(t)\dif t
   \le \int_{0}^{\infty} e^{\lambda t}  \sum_{k=0}^{\infty}  a^{(1)}_k e^{-t\ell_k} \varphi_k^{'}(s) \dif t
   =-(R_{\lambda } f)'(s),\quad s\in (0,1)
\end{align*} where $f(s)=s-1$ and $R_{\lambda}$ is the resolvent of $S$.  The last equality is again from Theorem XIII 4.3 of Dunford and Schwartz\cite[p.1332]{dunford-schwartz} since $R_{\lambda}f\in D(S)$ (see Pazy\cite[p.9]{pazy}).

$R_{\lambda}f\in D(S)$ also implies that $\abs{(R_{\lambda } f)'(s)}<\infty$ on any compact subinterval of $J$. Thus,
\begin{equation}\label{left}
   \int_{0}^{\infty} e^{\lambda t} P_{11}(t)\dif t<\infty,\qquad 0\leq \lambda<\ell_0,
\end{equation}
which implies that
 \begin{equation}
    \lambda_C=\sup\set{\lambda\geq 0:\,\int_{0}^{\infty} e^{\lambda t} P_{11}(t)\dif t<\infty }\geq \ell_0.
 \end{equation}
\end{proof}{\hfill\large{$\Box$}}
\par We are now ready to give the proof for our first main result stated in Section 2.
\\
\par \noindent{\it \textbf{Proof of Theorem~\ref{thm main}.}\,}
We give the proof by contradiction. Suppose $\lambda_C\neq  \ell_0$ then it follows from Lemma~\ref{lda great} that $\lambda_C>  \ell_0$. 
Denote $\tau = \inf\set{t\ge 0,\,X_t=0 }$  and  $x_i(t)=P_i(\tau>t)=\sum_{j\in \N} P_{ij}(t)$.

 Since the set $N_0=\set{i\in \N:\,q_{i0}>0}={1}$ is finite, the conclusion in Jacka and Roberts\cite{jacka} implies that
 \begin{align*}
\lambda_C =- \lim_{t\to \infty} \frac{\log x_1(t)}{t}.
 \end{align*} Thus,  for each $\epsilon>0$ such that $\ell_0 +\epsilon <\lambda_C$,  we obtain that when $t$ is large enough,  $$e^{ t (\ell_0 +\epsilon)}x_1(t)\le 1.$$
Hence,
\begin{equation*}
\lim_{t\to \infty}  e^{\ell_0 t}x_1(t)=\lim_{t\to \infty} \sum_{j\in \N}  e^{\ell_0 t} P_{1j}(t) =0,
\end{equation*} which implies that
\begin{equation}\label{limit 0}
    \lim_{t\to \infty}e^{\ell_0 t}[ P_{11}(t)+\sum_{j=2}^{\infty} j P_{1j}(t)s^{j-1}]=0,\qquad  s\in(0,1).
\end{equation}

On the other hand, it follows from Lemma~\ref{pijt expression} that for any $s\in(0,1)$,
\begin{align}
 0\le   \lim_{t\to \infty}e^{\ell_0 t}[ P_{11}(t)+\sum_{j=2}^{\infty}j P_{1j}(t)s^{j-1}]&= \lim_{t\to \infty} \sum_{k=0}^{\infty}  a^{(1)}_k e^{-t(\ell_k-\ell_0)} \varphi_k^{'}(s)\nonumber \\
   &=a^{(1)}_0 \varphi_0^{'}(s),\label{bigger o}
\end{align}where the last equality is from Lebesgue's dominated convergence theorem and the absolute convergence of the series $\sum_{k=0}^{\infty}  a^{(1)}_k \varphi_k^{'}(s)$.

By Lemma~\ref{spectra property}, we can take $\varphi_0(s)> 0, s\in(0,1)$. Hence,
\begin{equation}\label{ao1}
   a_0^{(1)}=\int_0^1 (s-1)\varphi_0(s) w(s)\dif s < 0.
\end{equation}
By combining Eqs.(\ref{bigger o}),(\ref{ao1}) with Eq.(\ref{limit 0}), we obtain that $\varphi_0^{'}(s)\equiv 0,\,s\in (0,1)$. Thus, $\varphi_0(s)\equiv constant$ in $(0,1)$ which is a contradiction to Corollary~\ref {coro e.200802}. Thus, the proof of Theorem $2.1$ is finished.
{\hfill\large{$\Box$}}

\subsection{Proof of Theorem \ref{prop bound b2}}
Denote $\GH$ the Hilbert space $L^2(J,\,w_1)$ with $w_1(s)=s$. Since $\frac{1}{w(s)},\,\frac{1}{w_1(s)}\in L^1_{loc}(J)$, the Cauchy-Schwarz inequality implies that $\FH,\,\GH\subset L^1_{loc}(J)$, the reader can refer to Corollary~1.6 of Kufner and Opic \cite{Kuf Opic 84} for details.

Denote $\mathcal{S}=\set{w(s),w_1(s)}$. Let us define the Sobolev space with weight $\mathcal{S}$
\begin{equation*}
W^{1,2}(J,\,\mathcal{S})
\end{equation*} as the set of all functions $f\in \FH$ such that their weak derivative
 (or say distributional derivative) $\mathrm{D}f$ are again the element of $\GH$. Theorem~1.11 of Kufner and Opic\cite{Kuf Opic 84} says that $W^{1,2}(J,\,\mathcal{S})$ is a Hilbert space if equipped with the norm
 $$|\|f|\|^2
 =\norm{f}_{\FH}^2+\norm{\mathrm{D}f}_{\GH}^2.$$

 Let $C_{c}^{\infty}(J)$ denote the space of infinitely differentiable functions $\phi:J\to \R$, with compact support in $J$. Since $w(s),\,w_1(s),\,\frac{1}{w(s)},\,\frac{1}{w_1(s)}\in L^1_{loc}(J)$, it follows from Lemma~4.4 of Kufner and Opic\cite{Kuf Opic 84} that $C_{c}^{\infty}(J)\subset W^{1,2}(J,\,\mathcal{S})$.  Then we define $$W^{1,2}_0(J,\,\mathcal{S})=\overline{C_{c}^{\infty}(J)},$$
 the closure being taken with respect to the norm of the weighted Sobolev space $W^{1,2}(J,\,\mathcal{S})$.

 Let $Q$ be the quadratic form defined on the domain $D_{min}'$ of the non-negative symmetric operator $S_{min}'$ by
\begin{align*}
Q (f,g)=\innp{S_{min}' f,\,g }=\int_0^1\, s f'(s)g'(s)\dif s.
\end{align*} By Friedrichs extension theorem (see Theorem 4.4.5 of Davies\cite{davies 1995}), the quadratic form $Q$ is closable.  Let $\bar{Q}$ be the closure of $Q$.
Since the domain $D(\bar{Q})$ of $\bar{Q}$ is the closure of  $D_{min}'$  with respect to the norm of the weighted Sobolev space $W^{1,2}(J,\,\mathcal{S})$, we have that
$$D(\bar{Q})=W^{1,2}_0(J,\,\mathcal{S}).$$

\begin{lemma}
$(S,\,D(S))$ is the Friedrichs extension of $(S_{min}, D_{{min}})$.
That is to say, $\bar{Q}$ is  the quadratic form arising from the non-negative self-adjoint operator $(S,\,D(S))$.
\end{lemma}

\begin{proof}
Let $(L,\,D(L))$ be the non-negative self-adjoint operator associated with the closed quadratic form $\bar{Q}$, we need only show that $ D(L) =D(S)$.

Since $(L,\,D(L))$ is a self-adjoint realization of $(S_{min}, D_{{min}})$,  there exist $a_1,\,a_2\in \R$ with $(a_1,\,a_2)\neq(0,0)$ such that
\begin{align*}
D(L)=\set{y+c\cdot(a_1v_1+a_2v_2):\,y\in D_{{min}},\,c\in \R},
\end{align*}where $v_1,\,v_2$ are given in Lemma~\ref{prop202008}.
On the other hand, we have that $$D(L)\subset D(L^{\frac12})= D(\bar{Q}).$$ Hence, $a_1v_1+a_2v_2 \in D(\bar{Q})\subset W^{1,2}(J,\,\mathcal{S})$, which implies that
\begin{align*}
a_1v'_1(s)+a_2v'_2(s)\in \GH,
\end{align*} i.e.,
\begin{align*}
\int_0^1 \big(a_1v'_1(s)+a_2v'_2(s)\big)^2 s\dif s<\infty.
\end{align*} Since $v_1'(s)\in C_c^{\infty}(J)$ and $v'_2(s)=\frac{1}{s}$, we see that $a_2=0$. Hence $ a_1\neq 0$ and $D(L)=D(S)$. 
\end{proof}{\hfill\large{$\Box$}}
\par We now provide a proof for our second main result stated in Section 2.\\
\par \noindent{\it \textbf{Proof of Theorem~\ref{prop bound b2}}.\,}
 Since $D(\bar{Q})=\overline{C_{c}^{\infty}(J)}$ with respect to the norm of the weighted Sobolev space $W^{1,2}(J,\,\mathcal{S})$, we see $C_{c}^{\infty}(J)$ is a core for $\bar{Q}$. The variational formulae (see Theorem 4.5.3 of Davis\cite{davies 1995}) implies that the first eigenvalue of $S$ can be expressed as
 \begin{align*}
 \ell_0&=\inf\set{ {Q}(f):\,\, f\in C_{c}^{\infty}(J),\,\norm{f}_{\FH}=1}\\
 &=\inf\set{\frac {\int_0^1 s\big( f'(s)\big)^2 \dif s}{\int_0^1 f^2(s) w(s)\dif s}:\,\, f\not\equiv 0,\,f\in C_{c}^{\infty}(J)}.
 \end{align*} Hence, we obtain (\ref{variational formula202008}).

 It is obvious that for any $\xi\in (0,1)$, the function $$f_{\xi}(s)=\int_s^1\frac{1}{ w_1(r)} \mathrm{1}_{(\xi, 1)}(r) \dif r,\quad s\in (0,1)$$
 belongs to the domain $D(S)$.
 Hardy inequality (see Theorem 6.2 of Opic and Kufner\cite[p.65]{obk}) implies that optimal constant $C$ of  the Hardy's inequality
 $$\big(\int_0^1 f^2(s) w(s)\dif s \big)^{\frac12} \le C \big(\int_0^1\big( f'(s)\big)^2  w_1(s) \dif s \big)^{\frac12},\quad   f(1)=0 $$
  satisfies the estimates
 $$ D \le C\le  2 D.$$
 where
 \begin{align*}
 D=\sup_{s\in (0,1)}\set{\Big( \int_0^s w(r)\dif r \Big)^{\frac12}\Big( \int_s^1 \frac{1}{w_1(r)}\dif r\Big)^{\frac12}}.
 \end{align*}
Hence, we obtain (\ref{lambda0 bound}) and (\ref{d2-square}). This ends the proof of Theorem 2.2. {\hfill\large{$\Box$}}\\

\section{Examples}
\setcounter{lemma}{0}
\renewcommand{\theex}{\arabic{section}.\arabic{ex}}
We now provide two examples to illustrate our results obtained in the previous section. The purpose of providing these two examples is two-fold: on the one hand, it shows that in some cases, the value of the Hardy index $D^2$ can be exactly given and on the other hand, these two examples  will be helpful in getting better bounds for estimating Hardy index values for general models, see the following Section 5.
\begin{ex}[ Quadratic birth--death process]\label{exmp1}
When  $b_j\equiv 0,\,\forall j\geq 3$, the quadratic branching process (\ref{branching rate}) degenerates to a birth-death process with the birth rate $\{\nu_n\}$ and death rate $\{\mu_n\}$ as follows:
\begin{equation*}\label{bd rate}
   \left\{
      \begin{array}{ll}
    \nu_n= b n^2, &\quad   \\
    \mu_n= a n^2 .  &\quad
      \end{array}
\right.
\end{equation*}
Here we have denoted $a=b_0$ and $b=b_2$, then  the condition $B'(1)<0$ means that $b<a$. Let $\kappa=\frac{b}{a}$. Although this is a well-discussed process, we may still get some new conclusions. In particular, for this special case, we can get the exact value of $D^2$ presented in (\ref{d2-square}). Indeed, it is fairly easy to show that(see below)
\begin{align}
   D^2&= \frac{1}{a-b}\sup_{s \in (0,1)} \set{(-\log s)(\log\frac{1- \kappa s}{ 1-s } )}\nonumber\\
   &=\frac{\big[\log( 1+\sqrt{1-\kappa} )\big]^2}{a-b},
\end{align}which then implies that
\begin{equation*}
\frac{a-b}{4\big[ \log(1+\sqrt{1-\kappa})\big]^2}\leq  \lambda_C\leq \frac{a-b}{\big[ \log(1+\sqrt{1-\kappa})\big]^2}.
\end{equation*}
When $b \to a^-$, the limit of the lower bound is $ \frac{a}{4}$, which is the exact value of the decay parameter $\lambda_C$ when $a=b$. See Chen\cite{cmf2010} or Roehner and Valent\cite{rv}.

Comparing our results with bounds obtained in Chen~\cite{cmf2010}, we find that our bound estimates are better than the estimates in Chen~\cite[Thorem 4.2]{cmf2010}
$$\frac1{4\delta}\leq \lambda_C\leq\frac1{\delta},$$ but
worse than the improved estimates in Chen~\cite[Corollary 4.4]{cmf2010}
$$\frac1{\delta_1}\leq \lambda_C \leq\frac1{\delta_1'}.$$
For more details of $\delta, \delta_1, \delta_1'$, we refer to Chen \cite[Section 4]{cmf2010}. \end{ex}

\par To obtain the exact value $D^2$ for our quadratic birth-death process, we need the following lemma.
\begin{lemma}\label{prop 03-6}
Suppose that $\sigma$ is a strictly positive constant. Then
$$\log(1+\sigma t)\log(1+ \frac{\sigma }{t})\leq [\log(1+\sigma)]^2, \qquad\forall t\in (0,\infty).$$
\end{lemma}
\begin{proof}
We maximize $f(t,s)=\log(1+\sigma t)\log(1+  {\sigma }{s}),\,(s,t)\in (0,\infty)\times (0,\infty) $ subject to the constraint $s-\frac{1}{t}=0$ using the method of Lagrange multipliers.
Let $$F(t,s,\theta)=\log(1+\sigma t)\log(1+  {\sigma }{s})+ (s-\frac{1}{t})\theta.$$
Then we have that
\begin{equation}\label{lagrang}
     \left\{
      \begin{array}{lll}
    \frac{\sigma }{1+\sigma t}\log(1+  {\sigma }{s})+\frac{\theta}{t^2}=0\\
    \frac{\sigma }{1+\sigma s}\log(1+  {\sigma }{t})+ {\theta} =0\\
    s-\frac{1}{t}=0,
      \end{array}
\right.
\end{equation}
which implies that
\begin{align}\label{ts}
   \frac{1+\sigma t}{t }\log(1+  \sigma  t)= \frac{1+\sigma s}{s}\log(1+  \sigma s).
\end{align}

Consider the function $G(t)=\frac{1+\sigma t}{t }\log(1+  \sigma  t)$. Since for $t\in (0,\infty)$,
\begin{align*}
   G'(t)=\frac{\sigma}{t}-\frac{1}{t^2} \log(1+  \sigma  t)=\frac{1}{t^2} [\sigma t -\log(1+  \sigma  t) ]>0
\end{align*}
and thus Eq.(\ref{ts}) implies that $t=s$. Together with the third equation of (\ref{lagrang}), we have that $t=s=1$ which implies the desired inequality.

We can also get the result by another more primary approach.

Let $f(t)=\log(1+\sigma t)\log(1+ \frac{\sigma }{t})$, then it is clear that $f(t)=f(\frac1t)$. By differentiating both side, we obtain
$$
f'(t)=-\frac{1}{t^2}f'(\frac1t),
$$
which implies that, if $f'(t)\geq0$ for $t\in (0,1)$, then $f'(t)\leq0$ for $t\in (1,\infty)$. Thus, the desired inequality follows  from $f(t)\leq f(1)$.

Hence it remains to show $f'(t)\geq0$ on $(0,1)$, which can be simplified to
\begin{equation}\label{midstep1}
(t^2+\sigma t)\log(1+\frac{\sigma}{t})\geq(1+\sigma t)\log(1+\sigma t),\qquad 0<t<1.
\end{equation}
Now consider $g(x)=(1+\sigma x)\log(1+\sigma x)$, and straight line $l(x)=(1+\sigma)\log(1+\sigma)x$, it's easy to see
$$
g(0)=l(0),\qquad g(1)=l(1).
$$
Hence, by the convexity of $g(x)$, we obtain
\begin{equation}\label{convex}
g(x)<l(x),\text{\quad for }x\in(0,1),\qquad g(x)>l(x)\text{ \quad for }x\in (1,\infty).
\end{equation}
Thus, to show (\ref{midstep1}), it suffices to show
\begin{equation}\label{midstep2}
  (t^2+\sigma t)\log(1+\frac{\sigma}{t})\geq l(t),\qquad 0<t<1.
\end{equation}
Letting $s=\frac1t$ yields that (\ref{midstep2}) is equivalent to
\begin{equation*}
  (1+\sigma s)\log(1+\sigma s)\geq l(s),\qquad 1<s<\infty.
\end{equation*}
That follows immediately from (\ref{convex}), which completes the proof.
\end{proof}{\hfill\large{$\Box$}}

Now we are ready to get the $D^2$-value for the quadratic birth-death process. Indeed, for $\kappa\in (0,1)$, taking $\sigma=\sqrt{1-\kappa}$ and $\frac{1}{x}=1+\sigma t $, we immediately obtain from Lemma~\ref{prop 03-6} that
\begin{align*}
 \sup_{x \in (0,1)} \set{-\log x\log\frac{1-\kappa x}{ 1-x } } &=\big[ \log( 1+\sqrt{1-\kappa} )\big]^2.
\end{align*}
Substituting the above indentity into (\ref{d2 0 new}), we have that
\begin{align}\label{eq 3.48}
D^2=\frac{\big[\log( 1+\sqrt{1-\kappa} )\big]^2}{(1-\kappa)a}.
\end{align}
Together with (\ref{lambda0 bound}) and the remarks before Theorem 2.2, the conclusions for Example \ref{exmp1} are obtained.
\begin{ex}[Quadratic branching process with upwardly skipping 2]

A Quadratic branching process is called with upwardly skipping 2 if $b_0>0, b_2\geq0, b_3>0,\,b_j\equiv 0,\,\forall j\geq 4$.
We are aware this case has not been discussed in the literatures yet. For this new case, we have
\begin{align*}
B(s)
&=(s-1)[b_3s^2 +(b_2+b_3)s- b_0].
\end{align*} Hence $B'(1)<0$ is equivalent to $b_2+2b_3<b_0$ which then implies that there are three real roots  $s_0, s_1, s_2$ of $B(s)=0$ such that $s_0=1$, $s_1>1$  and  $s_2 <0$.
Moreover, it is fairly easy to show that the function
\begin{equation*}
\phi(s)=(-\log s)(\int_0^s \frac{\dif r}{B (r)})
\end{equation*} is concave on $(0,1)$(see the following Lemma \ref{52lem}) and there is only one stationary point $s_0$ with $0<s_0<1$ of the function $\phi(s)$, i.e., $\phi'(s_0)=0$. Hence
\begin{equation*}
D^2=\sup_{s\in (0,1)} \phi(s)=\phi(s_0),
\end{equation*}and
\begin{align}\label{eq 4.7}
 \frac{1}{4 \phi(s_0) }\le \lambda_C\le \frac{1}{\phi(s_0) }.
\end{align}
\end{ex}

\par
For convenience, let's denote $B(s)$ as $B(x)$ and let $x_1=s_0$, $x_2=s_1$ and $x_3=s_2$ and thus $x_1=1$, $x_2=c>1$ and $x_3=-d$ with $d>c$ when  $b_3=1$.
\begin{lemma}\label{52lem}
  Let the above assumption of the cubic polynomial $B(x)$ prevail. Then we have that the function
  \begin{align}
    \varphi(x)=\big(\log \frac{1}{x}\big)\cdot\big( \int_0^x \frac{1}{B(t)}\dif t\big) ,\qquad x\in (0,1)
  \end{align}is concave on $(0,1)$.
\end{lemma}
\begin{proof}
   Without loss of any generality, we assume that $b_3=1$. Then $B(x)=(x-1)(x-c)(x+d)$. Hence, by the method of undetermined coefficients, we have the resolution to the partial fractions of the function $\frac{1}{B(x)}$:
   \begin{align*}
      \frac{1}{B(x)}&=\frac{\alpha_1}{x-1}+\frac{\alpha_2}{x-c}+\frac{\alpha_3}{x+d},\\
      \alpha_2&=\frac{1}{(c -1)(c+d)}>0,\,\,\alpha_3=\frac{1}{(d +1)(d+c)}>0,\\
      \alpha_1&=\frac{1}{(1-c)(1+d)}=-(\alpha_2+\alpha_3).
   \end{align*}
   Thus, for any $x\in (0,1)$,
   \begin{align*}
     \int_0^x \frac{1}{B(t)}\dif t &=\alpha_2\int_0^x \frac{1}{1-t}-\frac{1}{ c-t}\dif t +\alpha_3\int_0^x \frac{1}{ d+ t}+ \frac{1}{1-t}\dif t \\
     &=\alpha_2 \log \frac{1-\frac{x}{c}}{1-x}+\alpha_3 \log \frac{1+\frac{x}{d}}{1-x} .\\
   \end{align*}
It follows that  $$\varphi(x)= \alpha_2 \log\frac{1}{x}\log \frac{1-\frac{x}{c}}{1-x}+\alpha_3 \log\frac{1}{x}\log \frac{1+\frac{x}{d}}{1-x}.$$

   Since $\abs{\frac{1}{c}}<1$ and $\abs{\frac{1}{d}}<1$, it follows from Lemma~\ref{main lem} that both $\log\frac{1}{x}\log \frac{1-\frac{x}{c}}{1-x}$ and $\log\frac{1}{x}\log \frac{1+\frac{x}{d}}{1-x} $ are concave, which implies that $\varphi''(x)<0$ because of $\alpha_2,\,\alpha_3>0$.
\end{proof}{\hfill\large{$\Box$}}


The following simple inequality of the logarithm function is crucial to our future analysis and the proof of this inequality can be found in, say, Kuang's book ( \cite[Theorem 53, page 293 ]{kuang 3}).
\begin{proposition}\label{kuang 000} 
   If $x>0$ and $x\neq 1$ then
   \begin{equation}
      \frac{\log x}{x-1}\leq \frac{1+x}{2x}.
   \end{equation}
\end{proposition}
We also need an inequality about a univariate quadratic polynomial as follows. We ignore its proof since it is  very simple.
\begin{proposition}\label{quadratic ply}
   \begin{equation}
     p(1-p)x^2-4px+p-1<0,\qquad \forall x\in (0,1),\,\, \abs{p}<1
   \end{equation}
\end{proposition}

\begin{lemma}\label{main lem}
 Suppose that $\abs{p}<1$ is a fixed constant, then the function defined by $f(x)\triangleq-\log x\log\frac{1+p x}{ 1-x }$ is a concave function on $(0,1)$, i.e., $f''(x)<0$ on $(0,1)$.
\end{lemma}
\begin{proof}
By the Leibniz rule, we may easily compute the first and the second derivatives of function $f(x)$ as follows,
\begin{align}
  f'(x)&=-\frac{1}{x }\log \frac{1+p x}{ 1-x } +\frac{ 1+p }{(1-x)(1+px)}\log\frac{1}{x},\label{deri gx 1}\\
   f''(x)&=\frac{1}{x^2}\log \frac{1+p x}{ 1-x } -\frac{2(1+p)}{x(1-x)(1+px)}+\frac{(1+p)(2px+1-p)}{(1-x)^2(1+px)^2}\log\frac{1}{ x}. \label{deri gx}
 \end{align}
 It is easy to check that on $(0,1)$, the function $f(x)$ satisfies the following symmetric relationship
 \begin{equation}\label{symmetric}
   f(x)=f(\frac{1-x}{1+px}).
 \end{equation}
 Denote the transformation $$y=T(x)=\frac{1-x}{1+px}=\frac{1}{p}[\frac{1+p}{1+px}-1],\quad x\in(0,1).$$
By differentiating the symmetric equation (\ref{symmetric}), and using the chain rule and the product rule we immediately obtain that when $x\in (0,\, 1)$,
 \begin{align}
 f'(x)&=f'(y)\big|_{y=T(x)}\cdot \frac{-(1+p)}{(1+px)^2}\nonumber\\
    f''(x)&= f''(y)\big|_{y=T(x)}\cdot \big[\frac{(1+p)}{(1+px)^2}\big]^2+f'(y)\big|_{y=T(x)}\cdot\frac{2p(1+p)}{(1+px)^3}\nonumber\\
    &=\frac{ 1+p }{(1+px)^3}\cdot\big[\frac{1+p}{1+px}f''(y)+2p f'(y) \big]\big|_{y=T(x)}\nonumber\\
    &=\frac{ 1+p }{(1+px)^3}\cdot\big[(1+py)f''(y)+ 2p f'(y) \big]\big|_{y=T(x)}.\label{f2deriv transf}
 \end{align}
It follows form Eqs.(\ref{deri gx 1}) and (\ref{deri gx}) we have that
 \begin{align}\label{f sum f deri}
   &\quad (1+px)f''(x)+ 2p f'(x)\nonumber\\
   &=-\frac{2(1+p)}{x(1-x)}+\frac{1-px}{x^2}\log \frac{1+p x}{ 1-x }+ \frac{(1+p)^2}{(1-x)^2(1+px) }\log\frac{1}{ x}.
 \end{align}
It follows from Proposition~\ref{kuang 000} that for any $x\in(0,1)$ and $\abs{p}<1$,
 \begin{align*}
    \log\frac{1}{ x}& \leq (\frac{1}{x}-1)\frac{1+\frac{1}{x}}{\frac{2}{x}}=\frac{(1-x)(1+x)}{2x},\\
    \mbox{and~~}\log \frac{1+p x}{ 1-x }&\leq (\frac{1+p x}{ 1-x }-1)\frac{1+\frac{1+p x}{ 1-x }}{\frac{2 (1+p x)}{( 1-x) } }
    =\frac{(1+p)x}{1-x}\cdot\frac{2+(p-1)x}{2(1+px)}.
 \end{align*}
 Substituting the above two inequalities into Eq.(\ref{f sum f deri}) then yields
 \begin{align*}
    &\quad (1+px)f''(x)+ 2p f'(x)\\
    &\leq -\frac{2(1+p)}{x(1-x)}+\frac{1-px}{x }\frac{ 1+p  }{1-x}\cdot\frac{2+(p-1)x}{2(1+px)}+ \frac{(1+p)^2}{(1-x) (1+px) }\frac{ (1+x)}{2x}\\
    &=\frac{1+p}{2x(1-x)(1+px)}\big[p(1-p)x^2-4px+p-1 \big]\\
    &<0.\qquad \text{(by Lemma~\ref{quadratic ply})}
 \end{align*}
 Finally substituting the above inequality into Eq.(\ref{f2deriv transf}), we immediately obtain the desired $f''(x)<0$ on $(0,1)$.
\end{proof}{\hfill\large{$\Box$}}

 By Lemma \ref{52lem}, we see that $\varphi(x)$ is concave on $(0,1)$. It is not difficult to prove that there is only one point $x_0\in (0, 1)$ such that $\varphi'(x_0)=0$ and we also have
\begin{align}
   D^2=\sup_{x\in (0,1)}\varphi(x)= \varphi(x_0).
\end{align}
Therefore, for our second example we can give the estimation of $\lambda_C$ by the above equality and Theorem \ref{prop bound b2}. That is that
\begin{equation}\label{lambda0 bound*}
  \frac{1}{ 4\varphi(x_0)}\leq \lambda_C\leq  \frac{1}{ \varphi(x_0)}.
\end{equation}

\section{Estimation of Hardy Index\\(Proofs of Corollaries 2.1-2.4)}
\setcounter{lemma}{0}

\par The basic aim of this final section is to estimate the value of the Hardy index $D^2$ for our QMBP and to prove Corollaries \ref{coro 2.1}-\ref{coro 2.3} stated in Section 2. To achieve this aim we need the following simple yet useful lemma which reveals the deep properties of $A(s)$ which is defined as above (in (\ref{3.5}), say) as $A(s)=\frac{B(s)}{1-s}$.
\begin{lemma}\label{lem 5.14}
If $B'(1)<0$, then $A(s)$ is a positive bounded analytic function on $(0,1)$ whose derivatives are negative functions on $[0,1]$. In particular, $A(s)$ is strictly decreasing on $[0,1]$ with minimal value on $[0,1]$ as $A(1)=b_0-m_b$ and maximum value on $[0,1]$ as $A(0)=b_0$. Also $A(s)$ is concave on $(0,1)$.

\end{lemma}
\begin{proof}
  Under the condition $B'(1)<0$, we know that by Proposition \ref{prop 1.1}, $B(s)$ has no zero on $(0,1)$. It follows that as a power series, $B(s)$ is analytic on $(0,1)$ and thus so is the function $A(s)$. In particular, $A(s)$ is a continuous function of $s\in (0,1)$. Note that \begin{equation}\label{eq 8}
 \lim_{s\downarrow0}A(s)=b_0>0
 \end{equation} and
 \begin{equation}\label{eq 9}
 \lim_{s\uparrow 1}A(s)=\lim_{s\uparrow 1}\frac{B(s)}{1-s}=\frac{B'(1)}{(-1)}=(-1)B'(1)>0
\end{equation}
which is a finite value.

In short, $$ \lim_{s\downarrow 0}A(s)=b_0 \qquad \lim_{s\uparrow 1}A(s)=m_d-m_b=b_0-m_b.$$
It follows that $A(s)$ is positive and bounded on $[0,1]$.
We now show that $A(s)$ is strictly decreasing on $[0,1]$.

Note that for $s\in(0,1)$ we have
\begin{equation}\label{eq 13}
A(s)= B(s)\cdot\sum_{n=0}^{+\infty}s^n=\sum_{j=0}^{+\infty}b_j s^j\cdot\sum_{n=0}^{+\infty}s^n.
\end{equation}
Since $A(s)$ is analytic on $(0,1)$ and thus we may expand $A(s)$ as a power series on $(0,1)$,
\begin{equation}\label{eq 14}
A(s)=\sum_{n=0}^{+\infty}a_ns^n.
\end{equation}
Then by (\ref{eq 13}) we get
\begin{equation}\label{eq 15}
\forall~  0\le n<+\infty,\quad a_n=\sum_{k=0}^nb_k.
\end{equation}
Now by (\ref{eq 15}) and (\ref{assmup series bj}) we get that
\begin{equation}\label{eq 17}
\forall n\geq 1, \quad a_n<0.
\end{equation}
It follows that all the coefficients of all the derivatives of $A(s)$ are definitely non-positive (and usually negative except the trivial case of polynomial that many coefficients might be zero). Therefore, all the derivative functions of $A(s)$ are negative on $[0,1]$.
In particular $\forall s\in[0,1]$, $A'(s)<0$ and thus
\begin{equation}\label{eq 19}
A(s)\downdownarrows \mbox{on\quad} (0,1).
\end{equation}
Therefore
\begin{equation}\label{eq 20}
Min_{s\in[0,1]}A(s)=A(1)=b_0-m_b\equiv m_d-m_b\equiv(-1)B'(1)
\end{equation}
and
\begin{equation}\label{eq 21}
Max_{s\in[0,1]}A(s)=A(0)=b_0.
\end{equation}
 Thus for any $s\in(0,1)$ we have
\begin{equation}\label{eq 12}
0<b_0-m_b<A(s)<b_0<+\infty.
\end{equation}
The fact that $A(s)$ is concave on $(0,1)$ is also easily follows from the fact that $A''(s)<0$ for all $s\in(0,1)$.
\end{proof}{\hfill\large{$\Box$}}




Using the interesting and useful properties regarding $A(s)$ revealed in Lemma \ref{lem 5.14}, we are able to prove Corollaries 2.1-2.4 stated in Section 2.
\vspace{.8cm}
\par \noindent{\it \textbf{Proof of Corollary~\ref{coro 2.1}.}\,}
Note first that the Hardy index $D^2$ represented in (\ref{d2-square}) can be rewritten as

\begin{align}
   D^2= \sup_{x\in (0,1)} \set{\int_0^x \frac{1}{(1-t)A(t)}\dif t \int_x^1 \frac{1}{t} \dif t}.\label{b2 ex}
\end{align}
By Lemma \ref{lem 5.14} we know that $0<b_0-m_b\leq A(x)\leq b_0$ for all $x$ in $[0,1]$. It follows that
\begin{align*}
 \frac{1}{b_0}\int_0^x \frac{1}{1-t}\dif t \int_x^1 \frac{1}{t} \dif t\leq \int_0^x \frac{1}{(1-t)A(t)}\dif t \int_x^1 \frac{1}{t} \dif t\leq \frac{1}{b_0-m_b}\int_0^x \frac{1}{1-t}\dif t \int_x^1 \frac{1}{t} \dif t.
\end{align*}Clearly, we have that for all $x\in (0,1)$,
\begin{align*}
   \int_0^x \frac{1}{1-t}\dif t \int_x^1 \frac{1}{t} \dif t&=\log(1-x)\log x\\
   &\leq \frac14 \big(\log x(1-x) \big)^2\\
   &\leq \frac14 \big(\log \frac14 \big)^2= (\log 2)^2.
\end{align*}
Hence, we obtain (\ref{eq2.11}) which ends the proof of Corollary 2.1.{\hfill\large{$\Box$}}

\vspace{.8cm}
In order to show Corollary ~\ref{coro 2.2}, first recall for quadratic birth-death process (see Example 4.1 in Section 4), we have assumed that $b_0=a>0$, $b_2=b>0$, $b_j\equiv0$($\forall j\ge 3$) and thus $b_1=-(a+b)$. Then $B(x)=a-(a+b)x+bx^2=(1-x)(a-bx)$. Then the condition $B'(1)<0$ means that $b<a$. Denote $\kappa=\frac{b}{a} $. From Eq.(\ref{d2-square}), it is easy to see that
\begin{align}
   D^2&= \frac{1}{a-b}\sup_{x \in (0,1)} \set{-\log x\log\frac{1- \kappa x}{ 1-x } }.\label{d2 0 new}
\end{align}

\par \noindent{\it \textbf{Proof of Corollary~\ref{coro 2.2}.}\,} As proved in Lemma~\ref{lem 5.14}, $A(s)$ is strictly decreasing and concave on $[0,1]$. It follows directly that $A(s)$ is sandwiched between the secant line of $A(s)$ denoted by $y_1(s)$ and the tangent line of $A(s)$ denoted by $y_2(s)$ defined as follows on $[0,1]$,
\begin{equation}\label{5.13}
y_1(s)=-m_b\cdot s+b_0,
\end{equation}

\begin{equation}\label{5.14}
y_2(s)= -m_b\cdot s+A(s_0)+m_bs_0,
\end{equation}
where $s_0$ is determined by the equation $-m_b=A'(s_0)$ which guarantees that $0<s_0<1$.
To say more exactly, we have that
\begin{equation}\label{eq 5.18}
y_1(s)\le A(s)\le y_2(s)\mbox{~~~for all~~} s\in[0,1].\end{equation}
Using (\ref{d2 0 new}) and (\ref{eq 5.18}), we may easily get that
\begin{equation*}
\sup_{s\in(0,1)}(-\log s)\int_0^s \frac{\dif r}{(1-r)y_2(r)}\leq D^2\leq \sup_{s\in(0,1)}(-\log s)\int_0^s \frac{\dif r}{(1-r)y_1(r)}.
\end{equation*}
Substituting (\ref{5.13}) and (\ref{5.14}) into the above yields that
\begin{equation}\label{eq 5.16}
\sup_{s\in(0,1)}(-\log s)\int_0^s \frac{\dif r}{(1-r)(m_bs_0+A(s_0)-m_br)}\leq D^2\leq \sup_{s\in(0,1)}(-\log s)\int_0^s \frac{\dif r}{(1-r)(b_0-m_br)}.
\end{equation}
Now both the right-most and left-most terms of (\ref{eq 5.16}) are all in the format of $D^2$ in the quadratic birth-death process discussed in Example 4.1 and thus by using the conclusions obtained in Example 4.1 and a little algebra we immediately obtain (\ref{eq2.12}). Then (\ref{eq2.13}) immediately follows which ends the proof of Corollary ~\ref{coro 2.2}. {\hfill\large{$\Box$}}
\vspace{.8cm}
\par Using the similar idea as used in proving Corollary ~\ref{coro 2.2}, we may similarly prove Corollary ~\ref{coro 2.2*} as follows.

\par \noindent{\it \textbf{Proof of Corollary~\ref{coro 2.2*}.}\,}Recall $D^2\equiv\sup_{s\in(0,1)}(-\log s)\int_0^s \frac{\dif r}{(1-r)A(r)}$, where $A(s)$ is analytic on $(0,1)$ and thus $\exists\xi\in(0,1)$ such that
\begin{equation}\label{5.16}
A(s)=A(0)+A'(\xi)s,
\end{equation}
But as proved in Lemma~$5.1$, $A'(s)$ is (strictly) decreasing on $[0,1]$ and thus $A'(1)\leq A'(\xi)\leq A'(0)$. Considering that $A'(0)=-\sum_{j=2}^{+\infty}b_j$ and $A'(1)=-\frac12 B''(1)$ and noting that $A(0)=b_0$ we may easily get that
\begin{equation}\label{5.17}
b_0-\frac12B''(1)\cdot s\leq A(s)\leq b_0-(\sum_{j=2}^{+\infty}b_j)\cdot s.
\end{equation}
Since $B'(1)<0$, we get that $\frac{-A'(0)}{A(0)}=\frac{\sum_{j=2}^{+\infty}b_j}{b_0}<1$. Now using the similar method in proving Corollary ~\ref{coro 2.2} together with the conclusions obtained in Example 4.1, the righthand side of (\ref{eq2.12*}) can be easily obtained. Moreover, under the condition $B''(1)<2b_0$, we may use the conclusions obtained in Example 4.1 once again to show that the left-hand side of (\ref{eq2.12*}) is also true. The proof of conclusions in Corollary(\ref{coro 2.2*}) is finished.  {\hfill\large{$\Box$}}
\vspace{.8cm}

The basic idea in proving Corollaries ~\ref{coro 2.2} and ~\ref{coro 2.2*} is to sandwich the function $A(s)$ by two straight lines and then use the conclusions obtained in Example 4.1. Now the idea in proving the following Corollary~\ref{coro 2.3} is to sandwich the function $A(s)$ by two parabolas and then use the conclusions obtained in Example 4.2.
\\
\par \noindent{\it \textbf{Proof of Corollary~\ref{coro 2.3}.}\,}
Since $D^2\equiv\sup_{s\in(0,1)}(-\log s)\int_0^s \frac{\dif r}{(1-r)A(r)}$, where $A(s)$ is analytic on $(0,1)$ and thus $\exists\xi\in(0,1)$ such that
\begin{equation}\label{5.18}
A(s)=a_0+a_1s+\frac{A''(\xi)}{2}s^2,
\end{equation}
where $a_0=b_0$, $a_1=b_0+b_1<0$. However by Lemma 2.1 $A''(s)$ is strictly decreasing on $(0,1)$ and thus
\begin{equation}
A''(1)\leq A''(\xi)\leq A''(0),
\end{equation}
hence for all $s\in(0,1)$ we have
\begin{equation}
a_0+a_1s+\frac12A''(1)s^2\leq A(s)\leq a_0+a_1s+\frac12A''(0)s^2
\end{equation}
It is easy to see that $A'(s)=\sum_{n=1}^{+\infty}na_ns^{n-1}$ and $A''(s)=\sum_{n=2}^{+\infty}n(n-1)a_ns^{n-2}$ and thus
\begin{equation}
A''(0)=2a_2=2(b_0+b_1+b_2)<0,
\end{equation}
\begin{equation}
A''(1)=\sum_{n=2}^{+\infty}n(n-1)a_n=\sum_{n=2}^{+\infty}n(n-1)\sum_{k=0}^{n}b_k< 0.
\end{equation}
Now if we further assume that $A''(1)>-\infty$.
Then
\begin{equation}
-\infty<\sum_{n=2}^{+\infty}n(n-1)\sum_{k=0}^nb_k\leq A''(\xi)\leq 2(b_0+b_1+b_2)<0
\end{equation}
For notational convenience, denote
\begin{equation}
E(s)=a_0+a_1s+\frac12A''(0)s^2
\end{equation}
\begin{equation}
F(s)=a_0+a_1s+\frac12A''(1)s^2
\end{equation}
Then $\forall s\in(0,1)$,
$$(1-s)F(s)\leq B(s)\leq (1-s)E(s)$$ and thus
\begin{equation}
\sup_{s\in(0,1)}(-\log s)\int_0^s \frac{\dif r}{(1-r)F(r)}\leq D^2\leq \sup_{s\in(0,1)}(-\log s)\int_0^s \frac{\dif r}{(1-r)E(r)}.
\end{equation}
Now by using our result regrading Example 4.2 and the preliminary remark made before, we get that both the functions $(-\log s)\int_0^s \frac{\dif r}{(1-r)F(r)}$ and $(-\log s)\int_0^s \frac{\dif r}{(1-r)E(r)}$ are concave on $(0,1)$. It follows that, if we let
\begin{equation}
\phi_1(s)=(-\log s)\int_0^s \frac{\dif r}{(1-r)F(r)}
\end{equation}
and
\begin{equation}
\phi_2(s)=(-\log s)\int_0^s \frac{\dif r}{(1-r)E(r)},
\end{equation}
then there exist $s_1\in(0,1)$ and $s_2\in(0,1)$ such that
\begin{equation}
\sup_{s\in(0,1)}(-\log s)\int_0^s \frac{\dif r}{(1-r)F(r)}=\phi_1(s_1)
\end{equation}
and
\begin{equation}
\sup_{s\in(0,1)}(-\log s)\int_0^s \frac{\dif r}{(1-r)E(r)}=\phi_1(s_2).
\end{equation}
Then we get
\begin{equation}
\phi_1(s_1)\le D^2\le\phi_2(s_2)
\end{equation}
and then (\ref{eq 2.16}) and consequently (\ref{eq2.16}) can be obtained which ends the proof of Corollary 2.4. {\hfill\large{$\Box$}}

\section*{Acknowledgements}
The work of Yong Chen is supported by National Natural Science Foundation of China (No.11961033) and the work of
Wu-Jun Gao is supported  by National Natural Science Foundation of China (No. 11701265).


\begin{thebibliography}{10}




\bibitem{behz} Bailey P. B., Everitt W. N., Hinton D. B. and Zettl A. Some spectral properties of the Heun differential equation. {\it Operator methods in ordinary and partial differential equations, Volume 132 of the series Operator Theory: Advances and Applications }, pp 87-110, Birkh\"{a}user, Basel, 2002.






\bibitem{Chen2002a} Chen A.Y. Uniqueness and extinction properties of generalised Markov branching processes. {\it J. Math. Anal. Appl.} 274  (2002), 482-494.



\bibitem{cmf2010} Chen  M.F.  Speed of stability for birth--death processes.  Front. Math. China 5 (2010),  379-515.

\bibitem{davies 1995}Davies E.B. {\it Spectral Theory and Differential Operators. Cambridge Studies in Advanced Mathematics. 42}. Cambridge University Press, 1995.
\bibitem{dunford-schwartz} Dunford N. and Schwartz J. T. {\it Linear operators. Part II. Spectral theory. Selfadjoint operators in Hilbert space}. John Wiley \& Sons, Inc., New York, 1988.


\bibitem{hl}Hinton D. B. and Lewis R. T. Singular differential operators with spectra discrete and bounded below. Proc. Roy. Soc. Edinburgh Sect. A 84 (1979), no. 1-2, 117-134.

\bibitem{jacka}Jacka S. D. and Roberts G. O. Weak convergence of conditioned processes on a countable state space. J. Appl. Probab., 32:4 (1995), 902-916.

\bibitem{kuang 3}Kuang J.C.{\it Changyong budengshi. (Chinese) [Applied inequalities] third edition. With a preface by Shan Zhen Lu.} Shandong Science and Technology Press, 2003.

\bibitem{Kuf Opic 84} Kufner A. and Opic B.  How to define reasonably weighted Sobolev spaces, Comm. Math. Univ. Carolinae, 23 (3), 1984, 537-554.

\bibitem{lv1}Letessier J. and Valent G. The generating function method for quadratic asymptotically symmetric birth and death processes. SIAM J. Appl. Math. 44 (1984), no. 4, 773-783.

\bibitem{obk}Opic B. and Kufner A. {\it Hardy-Type Inequalities, Pitman Res. Notes Math. 219.} Longman Sci. \& Tech., Harlow, 1990.


\bibitem{pazy} Pazy A. {\it Semigroups of Linear Operators and Applications to Partial Differential Equations}, Springer, 1983.



\bibitem{rv}Roehner B. and Valent G. Solving the birth and death processes with quadratic asymptotically symmetric transition rates. SIAM J. Appl. Math. 42 (1982), no. 5, 1020-1046.

\bibitem{van pollett 11}van Doorn, E. A. and Pollett, P. K.  Quasi-stationary distributions for discrete-state models. European J. Oper. Res. 230 (2013), no. 1, 1-14.

\bibitem{az1} Zettl A. {\it Sturm-Liouville Theorey}. Mathematical Surveys and Monographs, 121. American Mathematical Society, Providence, RI, 2005.


\end{thebibliography}
\end{document}